\documentclass{amsart}
\usepackage{amsmath}
\usepackage{amsfonts}

\newtheorem{theorem}{Theorem}
\theoremstyle{plain}
\newtheorem{acknowledgement}{Acknowledgement}

\newtheorem{corollary}{Corollary}

\newtheorem{lemma}{Lemma}

\numberwithin{equation}{section}
\input{tcilatex}

\begin{document}
\title[$q$\textbf{-EULER NUMBERS AND POLYNOMIALS WITH WEIGHT }$\alpha $]{%
\textbf{On the families of }$q$\textbf{-Euler numbers and polynomials and
their applications}}
\author[\textbf{S. Arac\i }]{\textbf{Serkan Arac\i }}
\address{\textbf{University of Gaziantep, Faculty of Science and Arts,
Department of Mathematics, 27310 Gaziantep, TURKEY}}
\email{\textbf{mtsrkn@hotmail.com.tr}}
\author[\textbf{M. Acikgoz}]{\textbf{\bigskip Mehmet Acikgoz}}
\address{\bigskip \textbf{University of Gaziantep, Faculty of Science and
Arts, Department of Mathematics, 27310 Gaziantep, TURKEY}}
\email{\textbf{acikgoz@gantep.edu.tr}}
\author[\textbf{H. Jolany}]{\textbf{\bigskip Hassan Jolany}}
\address{\textbf{School of Mathematics, Statistics and Computer Science,
University of Tehran, Iran}}
\email{\textbf{hassan.jolany@khayam.ut.ac.ir}}

\begin{abstract}
In the present paper, we investigate special generalized $q$-Euler numbers
and polynomials. Some earlier results of T. Kim in terms of $q$-Euler
polynomials with weight $\alpha $ can be deduced. For presentation of our
formulas we apply the method of generating function and $p$-adic $q$%
-integral representation on $%
%TCIMACRO{\U{2124} }%
%BeginExpansion
\mathbb{Z}
%EndExpansion
_{p}$. We summarize our results as follows. In section 2, by using
combinatorial techniques we present two formulas for $q$-Euler numbers with
weight $\alpha $. In section 3, we derive distribution formula
(Multiplication Theorem) for Dirichlet type of $q$-Euler numbers and
polynomials with weight $\alpha $. Moreover we define partial Dirichlet type
zeta function and Dirichlet $q$-$L$-function, and obtain some interesting
combinatorial identities for interpolating our new definitions. In addition,
we derive behavior of the Dirichlet type of $q$-Euler $L$-function with
weight $\alpha $, $\mathcal{L}_{q}^{\chi }\left( s,x\mid \alpha \right) $ at 
$s=0$. Furthermore by using second kind stirling numbers, we obtain an
explicit formula for Dirichlet type $q$-Euler numbers with weight $\alpha $,
and $\beta $. Moreover a novel formula for $q$-Euler-Zeta function with
weight $\alpha $ in terms of nested series of $\widetilde{\zeta }%
_{E,q}\left( n\mid \alpha \right) $ is derived . In section 4, by
introducing $p$-adic Dirichlet type of $q$-Euler measure with weight $\alpha 
$, and $\beta $, we obtain some combinatorial relations, which interpolate
our previous results. In section 5, which is the main section of our paper.
As an application, we introduce a novel concept of dynamics of the zeros of
analytically continued $q$-Euler polynomials with weight $\alpha $.
\end{abstract}

\maketitle

\section{\textbf{Introduction}}

In this paper, we use notations like $%
%TCIMACRO{\U{2115} }%
%BeginExpansion
\mathbb{N}
%EndExpansion
$, $%
%TCIMACRO{\U{211d} }%
%BeginExpansion
\mathbb{R}
%EndExpansion
$ and $%
%TCIMACRO{\U{2102} }%
%BeginExpansion
\mathbb{C}
%EndExpansion
$, where $%
%TCIMACRO{\U{2115} }%
%BeginExpansion
\mathbb{N}
%EndExpansion
$ denotes the set of natural numbers, $%
%TCIMACRO{\U{211d} }%
%BeginExpansion
\mathbb{R}
%EndExpansion
$ denotes the field of real numbers and $%
%TCIMACRO{\U{2102} }%
%BeginExpansion
\mathbb{C}
%EndExpansion
$ also denotes the set of complex numbers. When one talks of $q$-extension, $%
q$ is variously considered as an indeterminate, a complex number or a $p$%
-adic number.

Throughout this paper, we will assume that $q\in 
%TCIMACRO{\U{2102} }%
%BeginExpansion
\mathbb{C}
%EndExpansion
$ with $\left\vert q\right\vert <1$. The $q$-symbol $\left[ x:q\right] $
denotes as 
\begin{equation*}
\left[ x:q\right] =\frac{q^{x}-1}{q-1}\text{.}
\end{equation*}

Originally, $q$-Euler numbers and polynomials were introduced by L. Carlitz
in 1948 and gave properties of this polynomials (see \cite{Carlitz}, \cite%
{Carlitz 1}). Recently, Taekyun Kim, by using $p$-adic $q$-integral in the $%
p $-adic integers ring, has added a weight to $q$-Bernoulli numbers and
polynomials and gave surprising and fascinating identities of them (for
details, see \cite{Kim 8}). The $q$-Bernoulli numbers and polynomials with
weight $\alpha $ are related to weighted $q$-Bernstein polynomials which is
shown by Kim (for details, see \cite{Kim 7}). These polynomials have
surprising properties in Analytic Numbers Theory and in $p$-adic analysis,
especially, in Mathematical physics. Ryoo also constructed $q$-Euler numbers
and polynomials with weight $\alpha $ and introduced some properties of $q$%
-Euler numbers and polynomials with weight $\alpha $ in "A note on the
weighted $q$-Euler numbers and polynomials with weight $\alpha $, Advanced
Studies Contemporary Mathematics 21 (2011), No. 1, 47-54."

Analytic continuation of $q$-Euler numbers and polynomials was investigated
by Kim in \cite{Kim 1}. In previous paper, Araci $et$ $al$. also considered
analytic continuation of weighted $q$-Genocchi numbers and polynomials and
introduced some interesting ideas (for detail, see \cite{Araci 3}). In this
article, we also specify analytic continuation of weighted $q$-Euler numbers
and polynomials. Also, we give some interesting identities by using
generating function of Ryoo's weighted $q$-Euler polynomials.\newline

Because in the literature of our present paper we use of $p$-adic Arithmetic
and $p$-adic numbers, so we need to give short review on $p$-adic numbers.
Historically the $p$-adic numbers were introduced by K. Hensel in 1908 in
his book Theorie der algebra\'{\i}schen Zahlen, Leipzig, 1908 (for more
informations on this subject, see \cite{Hensel}).

Let $p$ be a prime number, fixed once and for all. If $x$ is any rational
number other than 0, we can write $x$ in the form $x=p^{n}\frac{a}{b}$ ,
where $a,b\in 
%TCIMACRO{\U{2124} }%
%BeginExpansion
\mathbb{Z}
%EndExpansion
$ are relatively prime to $p$ and $n\in 
%TCIMACRO{\U{2124} }%
%BeginExpansion
\mathbb{Z}
%EndExpansion
$. We now define%
\begin{equation*}
|x|_{p}=p^{-n}\text{ and }|0|_{p}=0,\text{and }ord_{p}(x)=n\text{ and }%
ord_{p}(0)=+\infty \text{.}
\end{equation*}

They satisfy the following properties,%
\begin{gather*}
|x|_{p}\geq 0,\text{ }|x|_{p}=0\text{ if and only if }x=0 \\
|x+y|_{p}\leq \max \{|x|_{p},|y|_{p}\}\text{ }(\text{the strong triangle
inequality}) \\
\text{with } \\
|x+y|_{p}=\max \{|x|_{p},|y|_{p}\}\text{ if }|x|_{p}\neq |y|_{p}(\text{the
isosceles triangle principle}) \\
|x.y|_{p}=|x|_{p}.|y|_{p}
\end{gather*}

$|x|_{p}$ is called the $p$-adic valuation. Ostrowski proved that each
nontrivial valuation on the field of rational numbers is equivalent either
to the absolute value function or to some $p$-adic valuation. The completion
of the field $%
%TCIMACRO{\U{211a} }%
%BeginExpansion
\mathbb{Q}
%EndExpansion
$ of rational numbers with respect to the p-adic valuation $|.|_{p}$ is
called the field of $p$-adic numbers and will be denoted $%
%TCIMACRO{\U{211a} }%
%BeginExpansion
\mathbb{Q}
%EndExpansion
_{p}$. The set%
\begin{equation*}
%TCIMACRO{\U{2124} }%
%BeginExpansion
\mathbb{Z}
%EndExpansion
_{p}=\{x\in 
%TCIMACRO{\U{211a} }%
%BeginExpansion
\mathbb{Q}
%EndExpansion
_{p}\mid |x|_{p}\leq 1\}
\end{equation*}

is the ring of $p$-adic integers. It can be easily proved that each $p$-adic
number $x$ can be written in the form%
\begin{equation*}
x=\sum_{n=-f}^{\infty }a_{n}p^{n}
\end{equation*}

where each $a_{n}$ is one of the elements $0,1,\cdot \cdot \cdot ,p-1$, and $%
f\in 
%TCIMACRO{\U{2124} }%
%BeginExpansion
\mathbb{Z}
%EndExpansion
$. This is called the Hensel representation of $p$-adic numbers. With this
representation, one obtain for $x\in 
%TCIMACRO{\U{211a} }%
%BeginExpansion
\mathbb{Q}
%EndExpansion
_{p}$, $ord_{p}(x)=+\infty $ if $a_{i}=0$ for all $i$ and $%
ord_{p}(x)=min\{s|a_{s}\neq 0\}$, otherwise. \newline
Moreover we can write%
\begin{equation*}
|x|_{p}=p^{-ord_{p}(x)}\text{.}
\end{equation*}

\section{\textbf{Properties of the }$q$\textbf{-Euler Numbers and
polynomials with weight }$\protect\alpha $}

For $\alpha \in 
%TCIMACRO{\U{2115} }%
%BeginExpansion
\mathbb{N}
%EndExpansion
\tbigcup \left\{ 0\right\} $, the weighted $q$-Euler polynomials are given
as:

For $x\in 
%TCIMACRO{\U{2102} }%
%BeginExpansion
\mathbb{C}
%EndExpansion
$, 
\begin{equation}
\sum_{n=0}^{\infty }\widetilde{E}_{n,q}\left( x\mid \alpha \right) \frac{%
t^{n}}{n!}=\left[ 2:q\right] \sum_{n=0}^{\infty }\left( -1\right)
^{n}q^{n}e^{t\left[ n+x:q^{\alpha }\right] }\text{.}  \label{Equation 1}
\end{equation}

As a special case, substituting $x=0$ into (\ref{Equation 1}), $\widetilde{E}%
_{n,q}\left( 0\mid \alpha \right) :=\widetilde{E}_{n,q}\left( \alpha \right) 
$ are called weighted $q$-Euler numbers. By (\ref{Equation 1}), we readily
derive the following 
\begin{equation}
\widetilde{E}_{n,q}\left( x\mid \alpha \right) =\frac{\left[ 2:q\right] }{%
\left[ \alpha :q\right] ^{n}\left( 1-q\right) ^{n}}\sum_{l=0}^{n}\binom{n}{l}%
\left( -1\right) ^{l}\frac{q^{\alpha lx}}{1+q^{\alpha l+1}}\text{,}
\label{Equation 2}
\end{equation}

where $\binom{n}{l}$ is the binomial coefficient. By expression (\ref%
{Equation 1}), we see that%
\begin{equation}
\widetilde{E}_{n,q}\left( x\mid \alpha \right) =q^{-\alpha x}\left(
q^{\alpha x}\widetilde{E}_{q}\left( \alpha \right) +\left[ x:q^{\alpha }%
\right] \right) ^{n}\text{,}  \label{Equation 3}
\end{equation}

with the usual convention of replacing $\left( \widetilde{E}_{q}\left(
\alpha \right) \right) ^{n}$ by $\widetilde{E}_{n,q}\left( \alpha \right) $
(for details, see \cite{Ryoo}).

Let $\widetilde{H}_{q}^{\left( \alpha \right) }\left( x,t\right) $ be the
generating function of weighted $q$-Euler polynomials as follows:%
\begin{equation}
\widetilde{H}_{q}^{\left( \alpha \right) }\left( x,t\right)
=\sum_{n=0}^{\infty }\widetilde{E}_{n,q}\left( x\mid \alpha \right) \frac{%
t^{n}}{n!}\text{.}  \label{Equation 4}
\end{equation}

Then, we easily notice that%
\begin{equation}
\widetilde{H}_{q}^{\left( \alpha \right) }\left( x,t\right) =\left[ 2:q%
\right] \sum_{n=0}^{\infty }\left( -1\right) ^{n}q^{n}e^{t\left[
n+x:q^{\alpha }\right] }\text{.}  \label{Equation 5}
\end{equation}

From expressions (\ref{Equation 4}) and (\ref{Equation 5}), we procure the
followings:

For $k$ (=even) and $n,\alpha \in 
%TCIMACRO{\U{2115} }%
%BeginExpansion
\mathbb{N}
%EndExpansion
\tbigcup \left\{ 0\right\} $, we have 
\begin{equation}
\widetilde{E}_{n,q}\left( \alpha \right) -q^{k}\widetilde{E}_{n,q}\left(
k\mid \alpha \right) =\left[ 2:q\right] \sum_{l=0}^{k-1}\left( -1\right)
^{l}q^{l}\left[ l:q^{\alpha }\right] ^{n}\text{.}  \label{Equation 6}
\end{equation}

For $k$ (=odd) and $n,\alpha \in 
%TCIMACRO{\U{2115} }%
%BeginExpansion
\mathbb{N}
%EndExpansion
\tbigcup \left\{ 0\right\} $, we have 
\begin{equation}
q^{k}\widetilde{E}_{n,q}\left( k\mid \alpha \right) +\widetilde{E}%
_{n,q}\left( \alpha \right) =\left[ 2:q\right] \sum_{l=0}^{k-1}\left(
-1\right) ^{l}q^{l}\left[ l:q^{\alpha }\right] ^{n}\text{.}
\label{Equation 7}
\end{equation}

Via Eq. (\ref{Equation 5}), we easily obtain the following:%
\begin{equation}
\widetilde{E}_{n,q}\left( x\mid \alpha \right) =q^{-\alpha x}\sum_{k=0}^{n}%
\binom{n}{k}q^{\alpha kx}\widetilde{E}_{k,q}\left( \alpha \right) \left[
x:q^{\alpha }\right] ^{n-k}\text{.}  \label{Equation 8}
\end{equation}

From (\ref{Equation 6})-(\ref{Equation 8}), we get the following theorem.

\begin{theorem}
Let $k$ be even positive integer. Then we have%
\begin{gather}
\left[ 2:q\right] \sum_{l=0}^{k-1}\left( -1\right) ^{l}q^{l}\left[
l:q^{\alpha }\right] ^{n}  \label{Equation 9} \\
=\left( 1-q^{k\left( 1-\alpha +\alpha n\right) }\right) \widetilde{E}%
_{n,q}\left( \alpha \right) -q^{k\left( 1-\alpha \right) }\sum_{l=0}^{n-1}%
\binom{n}{l}q^{\alpha lk}\widetilde{E}_{l,q}\left( \alpha \right) \left[
k:q^{\alpha }\right] ^{n-l}\text{.}  \notag
\end{gather}
\end{theorem}

\begin{theorem}
Let $k$ be an odd positive integer. Then, we procure the following%
\begin{gather}
\left[ 2:q\right] \sum_{l=0}^{k-1}\left( -1\right) ^{l}q^{l}\left[
l:q^{\alpha }\right] ^{n}  \label{Equation 10} \\
=\left( q^{k\left( 1-\alpha +\alpha n\right) }+1\right) \widetilde{E}%
_{n,q}\left( \alpha \right) +q^{k\left( 1-\alpha \right) }\sum_{l=0}^{n-1}%
\binom{n}{l}q^{\alpha lk}\widetilde{E}_{l,q}\left( \alpha \right) \left[
k:q^{\alpha }\right] ^{n-l}\text{.}  \notag
\end{gather}
\end{theorem}

\section{$q$\textbf{-Euler-Zeta function with weight }$\protect\alpha $}

The familiar Euler polynomials are defined by%
\begin{equation}
\frac{2}{e^{t}+1}e^{xt}=\sum_{n=0}^{\infty }E_{n}\left( x\right) \frac{t^{n}%
}{n!},\text{ }\left\vert t\right\vert <\pi \text{ cf. \cite{kim 4}.}
\label{Equation 11}
\end{equation}

For $s\in 
%TCIMACRO{\U{2102} }%
%BeginExpansion
\mathbb{C}
%EndExpansion
$, $x\in 
%TCIMACRO{\U{211d} }%
%BeginExpansion
\mathbb{R}
%EndExpansion
$ with $0\leq x<1$, Euler-Zeta function is given by%
\begin{equation}
\zeta_E \left( s,x\right) =2\sum_{m=0}^{\infty }\frac{\left( -1\right) ^{m}}{%
\left( m+x\right) ^{s}}\text{, }  \label{Equation 12}
\end{equation}

and%
\begin{equation}
\zeta_E \left( s\right) =\sum_{m=1}^{\infty }\frac{\left( -1\right) ^{m}}{%
m^{s}}\text{.}  \label{Equation 13}
\end{equation}

By expressions (\ref{Equation 11}), (\ref{Equation 12}) and (\ref{Equation
13}), Euler-Zeta functions are related to the Euler numbers as follows: 
\begin{equation*}
\zeta_E \left( -n\right) =E_{n}\text{.}
\end{equation*}

Moreover, it is simple to see%
\begin{equation*}
\zeta_E \left( -n,x\right) =E_{n}\left( x\right) \text{.}
\end{equation*}

The weighted $q$-Euler Hurwitz-Zeta type function is defined by 
\begin{equation*}
\widetilde{\zeta }_{E,q}\left( s,x\mid \alpha \right) =\left[ 2:q\right]
\sum_{m=0}^{\infty }\frac{\left( -1\right) ^{m}q^{m}}{\left[ m+x:q^{\alpha }%
\right] ^{s}}\text{ .}
\end{equation*}

Similarly, weighted $q$-Euler-Zeta function is given by%
\begin{equation*}
\widetilde{\zeta }_{E,q}\left( s\mid \alpha \right) =\left[ 2:q\right]
\sum_{m=1}^{\infty }\frac{\left( -1\right) ^{m}q^{m}}{\left[ m:q^{\alpha }%
\right] ^{s}}\text{.}
\end{equation*}

For $n,\alpha \in 
%TCIMACRO{\U{2115} }%
%BeginExpansion
\mathbb{N}
%EndExpansion
\tbigcup \left\{ 0\right\} $, we have%
\begin{equation*}
\widetilde{\zeta }_{E,q}\left( -n\mid \alpha \right) =\widetilde{E}%
_{n,q}\left( \alpha \right) \text{ (see \cite{Ryoo}).}
\end{equation*}

We now consider the function $\widetilde{E}_{q}\left( n:\alpha \right) $ as
the analytic continuation of weighted $q$-Euler numbers. All the weighted $q$%
-Euler numbers agree with $\widetilde{E}_{q}\left( n:\alpha \right) $, the
analytic continuation of weighted $q$-Euler numbers evaluated at $n$. For $%
n\geq 0$, $\widetilde{E}_{q}\left( n:\alpha \right) =\widetilde{E}%
_{n,q}\left( \alpha \right) $.

We can now state $\widetilde{E}%
%TCIMACRO{\U{b4}}%
%BeginExpansion
{\acute{}}%
%EndExpansion
_{q}\left( s:\alpha \right) $ in terms of $\widetilde{\zeta }%
%TCIMACRO{\U{b4}}%
%BeginExpansion
{\acute{}}%
%EndExpansion
_{E,q}\left( s\mid \alpha \right) $, the derivative of $\widetilde{\zeta }%
_{E,q}\left( s:\alpha \right) $%
\begin{equation*}
\widetilde{E}_{q}\left( s:\alpha \right) =\widetilde{\zeta }_{E,q}\left(
-s\mid \alpha \right) \text{, }\widetilde{E}%
%TCIMACRO{\U{b4}}%
%BeginExpansion
{\acute{}}%
%EndExpansion
_{q}\left( s:\alpha \right) =\widetilde{\zeta }%
%TCIMACRO{\U{b4}}%
%BeginExpansion
{\acute{}}%
%EndExpansion
_{E,q}\left( -s\mid \alpha \right) \text{.}
\end{equation*}

For $n,\alpha \in 
%TCIMACRO{\U{2115} }%
%BeginExpansion
\mathbb{N}
%EndExpansion
\tbigcup \left\{ 0\right\} $ 
\begin{equation*}
\text{ }\widetilde{E}%
%TCIMACRO{\U{b4}}%
%BeginExpansion
{\acute{}}%
%EndExpansion
_{q}\left( 2n:\alpha \right) =\widetilde{\zeta }%
%TCIMACRO{\U{b4}}%
%BeginExpansion
{\acute{}}%
%EndExpansion
_{E,q}\left( -2n\mid \alpha \right) \text{.}
\end{equation*}

This is suitable for the differential of the functional equation and so
supports the coherence of \noindent $\widetilde{E}_{q}\left( s:\alpha
\right) $ and $\widetilde{E}%
%TCIMACRO{\U{b4}}%
%BeginExpansion
{\acute{}}%
%EndExpansion
_{q}\left( s:\alpha \right) $ with $\widetilde{E}_{n,q}\left( \alpha \right) 
$ and $\widetilde{\zeta }_{E,q}\left( s\mid \alpha \right) $. From the
analytic continuation of weighted $q$-Euler numbers, we derive as follows: 
\begin{equation*}
\widetilde{E}_{q}\left( s:\alpha \right) =\widetilde{\zeta }_{E,q}\left(
-s\mid \alpha \right) \text{ and }\widetilde{E}_{q}\left( -s:\alpha \right) =%
\widetilde{\zeta }_{E,q}\left( s\mid \alpha \right) \text{.}
\end{equation*}

Moreover, we derive the following:

For $n\in 
%TCIMACRO{\U{2115} }%
%BeginExpansion
\mathbb{N}
%EndExpansion
$%
\begin{equation*}
\widetilde{E}_{-n,q}\left( \alpha \right) =\widetilde{E}_{q}\left( -n:\alpha
\right) =\widetilde{\zeta }_{E,q}\left( n\mid \alpha \right) \text{.}
\end{equation*}

The curve $\widetilde{E}_{q}\left( s:a\right) $ runs through the points $%
\widetilde{E}_{-s,q}\left( \alpha \right) $ and grows $\sim n$
asymptotically $\left( -n\right) \rightarrow -\infty $. The curve $%
\widetilde{E}_{q}\left( s:a\right) $ runs through the point $\widetilde{E}%
_{q}\left( -n:a\right) $. Then, we procure the following:%
\begin{eqnarray*}
\lim_{n\rightarrow \infty }\widetilde{E}_{q}\left( -n:\alpha \right)
&=&\lim_{n\rightarrow \infty }\widetilde{\zeta }_{E,q}\left( n\mid \alpha
\right) =\lim_{n\rightarrow \infty }\left( \left[ 2:q\right]
\sum_{m=1}^{\infty }\frac{\left( -1\right) ^{m}q^{m}}{\left[ m:q^{\alpha }%
\right] ^{n}}\right) \\
&=&\lim_{n\rightarrow \infty }\left( -q\left[ 2:q\right] +\left[ 2:q\right]
\sum_{m=2}^{\infty }\frac{\left( -1\right) ^{m}q^{m}}{\left[ m:q^{\alpha }%
\right] ^{n}}\right) =-q^{2}\left[ 2:q^{-1}\right] \text{.}
\end{eqnarray*}%
From this, we note that

\begin{equation*}
\widetilde{E}_{q}\left( -n:\alpha \right) =\widetilde{\zeta }_{E,q}\left(
n\mid \alpha \right) \mapsto \widetilde{E}_{q}\left( -s:\alpha \right) =%
\widetilde{\zeta }_{E,q}\left( s\mid \alpha \right) \text{.}
\end{equation*}

\textsc{Notations: }Assume that $p$ be a fixed odd prime number. Throughout
this paper we use the following notations. By $%
%TCIMACRO{\U{2124} }%
%BeginExpansion
\mathbb{Z}
%EndExpansion
_{p}$ we denote the ring of $p$-adic rational integers, $%
%TCIMACRO{\U{211a} }%
%BeginExpansion
\mathbb{Q}
%EndExpansion
$ denotes the field of rational numbers, $%
%TCIMACRO{\U{211a} }%
%BeginExpansion
\mathbb{Q}
%EndExpansion
_{p}$ denotes the field of $p$-adic rational numbers, and $%
%TCIMACRO{\U{2102} }%
%BeginExpansion
\mathbb{C}
%EndExpansion
_{p}$ denotes the completion of algebraic closure of $%
%TCIMACRO{\U{211a} }%
%BeginExpansion
\mathbb{Q}
%EndExpansion
_{p}$. Let $%
%TCIMACRO{\U{2115} }%
%BeginExpansion
\mathbb{N}
%EndExpansion
$ be the set of natural numbers and $%
%TCIMACRO{\U{2115} }%
%BeginExpansion
\mathbb{N}
%EndExpansion
^{\ast }=%
%TCIMACRO{\U{2115} }%
%BeginExpansion
\mathbb{N}
%EndExpansion
\cup \left\{ 0\right\} $. The $p$-adic absolute value is defined by $%
\left\vert p\right\vert _{p}=\frac{1}{p}$. In this paper we assume $%
\left\vert q-1\right\vert _{p}<1$ as an indeterminate. Let $UD\left( 
%TCIMACRO{\U{2124} }%
%BeginExpansion
\mathbb{Z}
%EndExpansion
_{p}\right) $ be the space of uniformly differentiable functions on $%
%TCIMACRO{\U{2124} }%
%BeginExpansion
\mathbb{Z}
%EndExpansion
_{p}$. For a positive integer $d$ with $\left( d,p\right) =1$, set 
\begin{eqnarray*}
X &=&X_{d}=\lim_{\overleftarrow{n}}%
%TCIMACRO{\U{2124} }%
%BeginExpansion
\mathbb{Z}
%EndExpansion
/dp^{n}%
%TCIMACRO{\U{2124} }%
%BeginExpansion
\mathbb{Z}
%EndExpansion
\text{,} \\
X^{\ast } &=&\underset{\underset{\left( a,p\right) =1}{0<a<dp}}{\cup }a+dp%
%TCIMACRO{\U{2124} }%
%BeginExpansion
\mathbb{Z}
%EndExpansion
_{p}
\end{eqnarray*}

and%
\begin{equation*}
a+dp^{n}%
%TCIMACRO{\U{2124} }%
%BeginExpansion
\mathbb{Z}
%EndExpansion
_{p}=\left\{ x\in X\mid x\equiv a\left( \func{mod}dp^{n}\right) \right\} 
\text{,}
\end{equation*}

where $a\in 
%TCIMACRO{\U{2124} }%
%BeginExpansion
\mathbb{Z}
%EndExpansion
$ satisfies the condition $0\leq a<dp^{n}$.

Firstly, for introducing fermionic $p$-adic $q$-integral, we need some basic
information which we state here. A measure on $%
%TCIMACRO{\U{2124} }%
%BeginExpansion
\mathbb{Z}
%EndExpansion
_{p}$ with values in a $p$-adic Banach space $B$ is a continuous linear map%
\newline
\begin{equation*}
f\mapsto \int f(x)\mu =\int_{%
%TCIMACRO{\U{2124} }%
%BeginExpansion
\mathbb{Z}
%EndExpansion
_{p}}f(x)\mu (x)
\end{equation*}

from $C^{0}(%
%TCIMACRO{\U{2124} }%
%BeginExpansion
\mathbb{Z}
%EndExpansion
_{p},%
%TCIMACRO{\U{2102} }%
%BeginExpansion
\mathbb{C}
%EndExpansion
_{p})$, (continuous function on $%
%TCIMACRO{\U{2124} }%
%BeginExpansion
\mathbb{Z}
%EndExpansion
_{p}$ ) to $B$. We know that the set of locally constant functions from $%
%TCIMACRO{\U{2124} }%
%BeginExpansion
\mathbb{Z}
%EndExpansion
_{p}$ to $%
%TCIMACRO{\U{211a} }%
%BeginExpansion
\mathbb{Q}
%EndExpansion
_{p}$ is dense in $C^{0}(%
%TCIMACRO{\U{2124} }%
%BeginExpansion
\mathbb{Z}
%EndExpansion
_{p},%
%TCIMACRO{\U{2102} }%
%BeginExpansion
\mathbb{C}
%EndExpansion
_{p})$ so.\newline

Explicitly, for all $f\in C^{0}(%
%TCIMACRO{\U{2124} }%
%BeginExpansion
\mathbb{Z}
%EndExpansion
_{p},%
%TCIMACRO{\U{2102} }%
%BeginExpansion
\mathbb{C}
%EndExpansion
_{p})$, the locally constant functions\newline
\begin{equation*}
f_{n}=\sum_{i=0}^{p^{n}-1}f(i)1_{i+p^{n}%
%TCIMACRO{\U{2124} }%
%BeginExpansion
\mathbb{Z}
%EndExpansion
_{p}}\rightarrow f\text{ in }C^{0}
\end{equation*}

Now, set $\mu (i+p^{n}%
%TCIMACRO{\U{2124} }%
%BeginExpansion
\mathbb{Z}
%EndExpansion
_{p})=\int_{%
%TCIMACRO{\U{2124} }%
%BeginExpansion
\mathbb{Z}
%EndExpansion
_{p}}1_{i+p^{n}%
%TCIMACRO{\U{2124} }%
%BeginExpansion
\mathbb{Z}
%EndExpansion
_{p}}\mu $. Then $\int_{%
%TCIMACRO{\U{2124} }%
%BeginExpansion
\mathbb{Z}
%EndExpansion
_{p}}f\mu $, is given by the following Riemannian sum%
\begin{equation*}
\int_{%
%TCIMACRO{\U{2124} }%
%BeginExpansion
\mathbb{Z}
%EndExpansion
_{p}}f\mu =\lim_{n\rightarrow \infty }\sum_{i=0}^{p^{n}-1}f(i)\mu {(i+p^{n}%
%TCIMACRO{\U{2124} }%
%BeginExpansion
\mathbb{Z}
%EndExpansion
_{p})}
\end{equation*}

T. Kim introduced $\mu $ as follows:%
\begin{equation*}
\mu _{-q}(a+p^{n}%
%TCIMACRO{\U{2124} }%
%BeginExpansion
\mathbb{Z}
%EndExpansion
_{p})=\frac{(-q)^{a}}{[p^{n}]_{-q}}
\end{equation*}

So, for $f\in UD\left( 
%TCIMACRO{\U{2124} }%
%BeginExpansion
\mathbb{Z}
%EndExpansion
_{p}\right) $, the fermionic $p$-adic $q$-integral on $%
%TCIMACRO{\U{2124} }%
%BeginExpansion
\mathbb{Z}
%EndExpansion
_{p}$ is defined by Kim as follows:%
\begin{eqnarray}
I_{-q}\left( f\right) &=&\int_{%
%TCIMACRO{\U{2124} }%
%BeginExpansion
\mathbb{Z}
%EndExpansion
_{p}}f\left( \eta \right) d\mu _{-q}\left( \eta \right)  \label{equation 1}
\\
&=&\lim_{n\rightarrow \infty }\frac{1}{\left[ p^{n}:-q\right] }\sum_{\eta
=0}^{p^{n}-1}q^{\eta }f\left( \eta \right) \left( -1\right) ^{\eta }\text{.}
\notag
\end{eqnarray}

Let $\chi $ be the Dirichlet's character with conductor $d$ (= odd)$\in 
%TCIMACRO{\U{2115} }%
%BeginExpansion
\mathbb{N}
%EndExpansion
$ and let us take $f\left( \eta \right) =\chi \left( \eta \right) \left[
x+\eta :q^{\alpha }\right] ^{n}$, then we define Dirichlet's type of $q$%
-Euler numbers and polynomials with weight $\alpha $ as follows:%
\begin{equation}
\widetilde{\mathcal{E}}_{n,q}^{\chi }\left( x\mid \alpha \right) =\int_{%
%TCIMACRO{\U{2124} }%
%BeginExpansion
\mathbb{Z}
%EndExpansion
_{p}}\chi \left( \eta \right) \left[ x+\eta :q^{\alpha }\right] ^{n}d\mu
_{-q}\left( \eta \right) \text{.}  \label{Equation 14}
\end{equation}

From (\ref{equation 1}), we have the following well-known equality.%
\begin{equation}
q^{d}\int_{%
%TCIMACRO{\U{2124} }%
%BeginExpansion
\mathbb{Z}
%EndExpansion
_{p}}f\left( \eta +d\right) d\mu _{-q}\left( \eta \right) +\left( -1\right)
^{d-1}\int_{%
%TCIMACRO{\U{2124} }%
%BeginExpansion
\mathbb{Z}
%EndExpansion
_{p}}f\left( \eta \right) d\mu _{-q}\left( \eta \right) =\left[ 2:q\right]
\sum_{l=0}^{d-1}q^{l}\left( -1\right) ^{d-1-l}f\left( l\right) \text{.}
\label{Equation 15}
\end{equation}

By expressions of (\ref{Equation 14}) and (\ref{Equation 15}), for $d$ $($%
=odd$)$ positive integer, we have the following%
\begin{equation}
\widetilde{\mathcal{E}}_{n,q}^{\chi }\left( x\mid \alpha \right) =\frac{%
\left[ 2:q\right] }{\left[ \alpha :q\right] ^{n}\left( 1-q\right) ^{n}}%
\sum_{j=0}^{n}\binom{n}{j}\left( -1\right) ^{j}q^{\alpha
jx}\sum_{l=0}^{d-1}\chi \left( l\right) q^{l}\left( -1\right) ^{l}\frac{%
q^{\alpha jl}}{q^{\left( \alpha j+1\right) d}+1}\text{.}  \label{Equation 16}
\end{equation}

Substituting $x=0$ in (\ref{Equation 16}), $\widetilde{\mathcal{E}}%
_{n,q}^{\chi }\left( 0\mid \alpha \right) :=$ $\widetilde{\mathcal{E}}%
_{n,q}^{\chi }\left( \alpha \right) $ are called Dirichlet type of $q$-Euler
numbers with weight $\alpha $. That is, we easily derive the following%
\begin{equation}
\widetilde{\mathcal{E}}_{n,q}^{\chi }\left( \alpha \right) =\frac{\left[ 2:q%
\right] }{\left( 1-q^{\alpha }\right) ^{n}}\sum_{j=0}^{n}\binom{n}{j}\left(
-1\right) ^{j}\sum_{l=0}^{d-1}\chi \left( l\right) \left( -1\right) ^{l}%
\frac{q^{\left( \alpha j+1\right) l}}{q^{\left( \alpha j+1\right) d}+1}\text{%
.}  \label{Equation 17}
\end{equation}

\begin{theorem}
Let $\chi $ be Dirichlet's character and for any $n\in 
%TCIMACRO{\U{2115} }%
%BeginExpansion
\mathbb{N}
%EndExpansion
^{\ast }$. Then we have%
\begin{equation*}
\widetilde{\mathcal{E}}_{n,q}^{\chi }\left( x\mid \alpha \right)
=\sum_{k=0}^{n}\binom{n}{k}q^{\alpha kx}\widetilde{\mathcal{E}}_{k,q}^{\chi
}\left( \alpha \right) \left[ x:q^{\alpha }\right] ^{n-k}\text{.}
\end{equation*}
\end{theorem}

\begin{proof}
By using (\ref{Equation 14}) and (\ref{Equation 17}), becomes%
\begin{eqnarray*}
\widetilde{\mathcal{E}}_{n,q}^{\chi }\left( x\mid \alpha \right)  &=&\int_{%
%TCIMACRO{\U{2124} }%
%BeginExpansion
\mathbb{Z}
%EndExpansion
_{p}}\chi \left( \eta \right) \left[ x+\eta :q^{\alpha }\right] ^{n}d\mu
_{-q}\left( \eta \right)  \\
&=&\int_{%
%TCIMACRO{\U{2124} }%
%BeginExpansion
\mathbb{Z}
%EndExpansion
_{p}}\chi \left( \eta \right) \left( \left[ x:q^{\alpha }\right] +q^{\alpha
x}\left[ \eta :q^{\alpha }\right] \right) ^{n}d\mu _{-q}\left( \eta \right) 
\text{.}
\end{eqnarray*}

From this, by using binomial theorem, we can write the following%
\begin{eqnarray*}
&&\sum_{k=0}^{n}\binom{n}{k}q^{\alpha kx}\left[ x:q^{\alpha }\right]
^{n-k}\int_{%
%TCIMACRO{\U{2124} }%
%BeginExpansion
\mathbb{Z}
%EndExpansion
_{p}}\chi \left( \eta \right) \left[ \eta :q^{\alpha }\right] ^{k}d\mu
_{-q}\left( \eta \right) \\
&=&\sum_{k=0}^{n}\binom{n}{k}q^{\alpha kx}\left[ x:q^{\alpha }\right] ^{n-k}%
\widetilde{\mathcal{E}}_{k,q}^{\chi }\left( \alpha \right) \text{.}
\end{eqnarray*}

Thus, we complete the proof of the theorem.
\end{proof}

\begin{theorem}
The following identity%
\begin{equation*}
\widetilde{\mathcal{E}}_{n,q}^{\chi }\left( dx\mid \alpha \right) =\frac{%
\left[ d:q^{\alpha }\right] }{\left[ d:-q\right] }\sum_{a=0}^{d-1}\left(
-1\right) ^{a}\chi \left( a\right) q^{a}\widetilde{E}_{n,q^{d}}\left( x+%
\frac{a}{d}\mid \alpha \right)
\end{equation*}%
holds true.
\end{theorem}

\begin{proof}
To prove this, we compute as follows:%
\begin{eqnarray*}
&=&\lim_{n\rightarrow \infty }\frac{1}{\left[ dp^{n}:-q\right] }%
\sum_{y=0}^{dp^{n}-1}\left( -q\right) ^{y}\chi \left( y\right) \left[
x+y:q^{\alpha }\right] ^{n} \\
&=&\frac{1}{\left[ d:-q\right] }\lim_{n\rightarrow \infty }\frac{1}{\left[
p^{n}:-q^{d}\right] }\sum_{y=0}^{p^{n}-1}\sum_{a=0}^{d-1}\left( -q\right)
^{a+dy}\chi \left( a+dy\right) \left[ x+a+dy:q^{\alpha }\right] ^{n} \\
&=&\frac{\left[ d:q^{\alpha }\right] }{\left[ d:-q\right] }%
\sum_{a=0}^{d-1}\left( -q\right) ^{a}\chi \left( a\right) \lim_{n\rightarrow
\infty }\frac{1}{\left[ p^{n}\right] _{-q^{d}}}\sum_{y=0}^{p^{n}-1}\left(
-q\right) ^{dy}\left[ \frac{x+a}{d}+y:q^{d\alpha }\right] ^{n} \\
&=&\frac{\left[ d:q^{\alpha }\right] }{\left[ d:-q\right] }%
\sum_{a=0}^{d-1}\left( -q\right) ^{a}\chi \left( a\right) \widetilde{E}%
_{n,q^{d}}\left( \frac{x+a}{d}\mid \alpha \right) \text{.}
\end{eqnarray*}

So, we get the desired result and proof is complete.
\end{proof}

By (\ref{Equation 16}), we procure the following:%
\begin{equation}
\sum_{n=0}^{\infty }\widetilde{\mathcal{E}}_{n,q}^{\chi }\left( x\mid \alpha
\right) \frac{t^{n}}{n!}=\left[ 2:q\right] \sum_{m=0}^{\infty }q^{m}\chi
\left( m\right) \left( -1\right) ^{m}e^{t\left[ x+m:q^{\alpha }\right] }%
\text{.}  \label{Equation 19}
\end{equation}

By applying derivative operator of order $k$ as $\frac{d^{k}}{dt^{k}}\mid
_{t=0}$, we have the following%
\begin{equation*}
\widetilde{\mathcal{E}}_{k,q}^{\chi }\left( x\mid \alpha \right) =\left[ 2:q%
\right] \sum_{m=0}^{\infty }q^{m}\chi \left( m\right) \left( -1\right) ^{m}%
\left[ x+m:q^{\alpha }\right] ^{k}\text{.}
\end{equation*}

That is, we can define Dirichlet $q$-$L$-function as follows:%
\begin{equation}
\mathcal{L}_{q}^{\chi }\left( s,x\mid \alpha \right) =\left[ 2:q\right]
\sum_{m=0}^{\infty }\frac{q^{m}\chi \left( m\right) \left( -1\right) ^{m}}{%
\left[ x+m:q^{\alpha }\right] ^{s}}.  \label{Equation 18}
\end{equation}

\begin{lemma}
The following equality holds true:%
\begin{equation*}
\mathcal{L}_{q}^{\chi }\left( -k,x\mid \alpha \right) =\widetilde{\mathcal{E}%
}_{k,q}^{\chi }\left( x\mid \alpha \right) \text{.}
\end{equation*}
\end{lemma}

\begin{proof}
Substituting $s=-k$ into (\ref{Equation 18}), we arrive at the desired
result.
\end{proof}

Now also, we define partial Dirichlet type zeta function as follows:%
\begin{equation}
\mathcal{H}_{q}^{\chi }\left( s:x:a:F\mid \alpha \right) =\left[ 2:q\right]
\sum_{m\equiv a\left( \func{mod}F\right) }^{\infty }\frac{q^{m}\chi \left(
m\right) \left( -1\right) ^{m}}{\left[ x+m:q^{\alpha }\right] ^{s}}\text{.}
\label{Equation 20}
\end{equation}

Now, for interpolating partial Dirichlet type zeta function, we rewrite it
in terms of weighted $q$-Euler Hurwitz-Zeta function as follows.

\begin{theorem}
For $F\equiv 1(\func{mod}2)$, then the following equality holds true:%
\begin{equation}
\mathcal{H}_{q}^{\chi }\left( s:x:a:F\mid \alpha \right) =\frac{\left[ 2:q%
\right] q^{a}\left( -1\right) ^{a}\chi \left( a\right) }{\left[ F:q^{\alpha }%
\right] ^{s}}\widetilde{\zeta }_{E,q^{F}}\left( s,\frac{x+a}{F}\mid \alpha
\right) \text{.}  \label{Equation 21}
\end{equation}
\end{theorem}

\begin{proof}
By expression of (\ref{Equation 20}), we compute as follows:%
\begin{eqnarray*}
\mathcal{H}_{q}^{\chi }\left( s:x:a:F\mid \alpha \right)  &=&\left[ 2:q%
\right] \sum_{m\equiv a\left( \func{mod}F\right) }^{\infty }\frac{q^{m}\chi
\left( m\right) \left( -1\right) ^{m}}{\left[ x+m:q^{\alpha }\right] ^{s}} \\
&=&\left[ 2:q\right] \sum_{m=0}^{\infty }\frac{q^{mF+a}\chi \left(
mF+a\right) \left( -1\right) ^{mF+a}}{\left[ x+mF+a:q^{\alpha }\right] ^{s}}
\\
&=&\frac{\left[ 2:q\right] q^{a}\left( -1\right) ^{a}\chi \left( a\right) }{%
\left[ F:q^{\alpha }\right] ^{s}}\sum_{m=0}^{\infty }\frac{\left(
q^{F}\right) ^{m}\left( -1\right) ^{m}}{\left[ \frac{x+a}{F}+m:q^{F\alpha }%
\right] ^{s}}\text{.}
\end{eqnarray*}

Thus, we arrive at the desired result.
\end{proof}

If we put $s=-n$ into (\ref{Equation 21}), then, we can write partial
Dirichlet type Zeta function in terms of weighted $q$-Euler numbers%
\begin{equation}
\mathcal{H}_{q}^{\chi }\left( -n:x:a:F\mid \alpha \right) =\left[ 2:q\right]
q^{a}\left( -1\right) ^{a}\chi \left( a\right) \left[ F:q^{\alpha }\right]
^{n}\widetilde{E}_{n,q^{F}}\left( \frac{x+a}{F}\mid \alpha \right) \text{.}
\label{Equation 22}
\end{equation}

\begin{theorem}
The following identity%
\begin{equation}
\mathcal{H}_{q}^{\chi }\left( s:x:a:F\mid \alpha \right) =\frac{\left[ 2:q%
\right] q^{a}\left( -1\right) ^{a}\chi \left( a\right) }{\left[
x+a:q^{\alpha }\right] ^{s}}\sum_{k=0}^{\infty }q^{\alpha k\left( x+a\right)
}\binom{-s}{k}\left( \frac{\left[ F:q^{\alpha }\right] }{\left[
x+a:q^{\alpha }\right] }\right) ^{k}\widetilde{E}_{k,q^{F}}
\label{Equation 23}
\end{equation}%
holds true.
\end{theorem}

\begin{proof}
Taking $n=-s$ into (\ref{Equation 22}) and some manipulation by using
combinatorial techniques, we can reach to the proof of the theorem.
\end{proof}

If we substitute $s=-n$ into (\ref{Equation 23}), then, (\ref{Equation 23})
reduces to (\ref{Equation 22}). \ Now also, we give the following theorem.

\begin{theorem}
Let $d\equiv 1(\func{mod}2)$, then, we have%
\begin{equation*}
\mathcal{L}_{q}^{\chi }\left( s,x\mid \alpha \right) =\frac{\left[ 2:q\right]
}{\left[ 2:q^{d}\right] \left[ d:q^{\alpha }\right] ^{s}}\sum_{l=0}^{d-1}%
\left( -1\right) ^{l}\chi \left( l\right) q^{l}\widetilde{\zeta }%
_{E,q^{d}}\left( s,\frac{x+l}{d}\mid \alpha \right) \text{.}
\end{equation*}
\end{theorem}

\begin{proof}
By using (\ref{Equation 18}),\ we compute as follows:%
\begin{eqnarray*}
\mathcal{L}_{q}^{\chi }\left( s,x\mid \alpha \right) &=&\left[ 2:q\right]
\sum_{m=0}^{\infty }\sum_{l=0}^{d-1}\frac{q^{l+md}\chi \left( l+md\right)
\left( -1\right) ^{l+md}}{\left[ x+l+md:q^{\alpha }\right] ^{s}} \\
&=&\frac{\left[ 2:q\right] }{\left[ 2:q^{d}\right] \left[ d:q^{\alpha }%
\right] ^{s}}\sum_{l=0}^{d-1}\left( -1\right) ^{l}\chi \left( l\right) q^{l}%
\widetilde{\zeta }_{E,q^{d}}\left( s,\frac{x+l}{d}\mid \alpha \right) \text{.%
}
\end{eqnarray*}%
Thus, we prove the above theorem.
\end{proof}

By means of the above theorem and (\ref{Equation 21}), we have the following
corollary.

\begin{corollary}
The following equality 
\begin{equation}
\mathcal{L}_{q}^{\chi }\left( s,x\mid \alpha \right) =\frac{1}{\left[ 2:q^{d}%
\right] }\sum_{l=0}^{d-1}\mathcal{H}_{q}^{\chi }\left( s:x:l:d\mid \alpha
\right)  \label{Equation 24}
\end{equation}%
holds true.
\end{corollary}

By (\ref{Equation 23}) and (\ref{Equation 24}), we have the following
corollary.

\begin{corollary}
The following nice identity%
\begin{equation*}
\mathcal{L}_{q}^{\chi }\left( s,x\mid \alpha \right) =\frac{\left[ 2:q\right]
}{\left[ 2:q^{d}\right] }\sum_{l=0}^{d-1}\frac{\left( -1\right) ^{l}\chi
\left( l\right) }{\left[ x+l:q^{\alpha }\right] ^{s}}\sum_{k=0}^{\infty
}q^{\alpha k\left( x+l\right) +l}\binom{-s}{k}\widetilde{E}_{k,q^{F}}\left( 
\frac{\left[ F:q^{\alpha }\right] }{\left[ x+l:q^{\alpha }\right] }\right)
^{k}
\end{equation*}%
holds true.
\end{corollary}

By using (\ref{Equation 22}) and (\ref{Equation 24}), we derive behavior of
the Dirichlet type of $q$-Euler $L$-function with weight $\alpha $ at $s=0$
as follows:

\begin{theorem}
The following identity holds true:%
\begin{equation*}
\mathcal{L}_{q}^{\chi }\left( 0,x\mid \alpha \right) =\frac{\left[ 2:q\right]
}{\left[ 2:q^{d}\right] }\sum_{l=0}^{d-1}\left( -1\right) ^{l}\chi \left(
l\right) q^{l}\text{.}
\end{equation*}
\end{theorem}

Now also, we define Dirichlet type $q$-Euler polynomials with weight $\alpha 
$ and $\beta $ with the following expression%
\begin{equation}
\widetilde{\mathcal{E}}_{n,q}^{\chi }\left( x\mid \alpha :\beta \right)
=\int_{%
%TCIMACRO{\U{2124} }%
%BeginExpansion
\mathbb{Z}
%EndExpansion
_{p}}\left[ x+\eta :q^{\alpha }\right] ^{n}\chi \left( \eta \right) d\mu
_{-q^{\beta }}\left( \eta \right) \text{.}  \label{Equation 25}
\end{equation}

Taking $x=0$ into (\ref{Equation 25}), we have $\widetilde{\mathcal{E}}%
_{n,q}^{\chi }\left( 0\mid \alpha :\beta \right) :=\widetilde{\mathcal{E}}%
_{n,q}^{\chi }\left( \alpha :\beta \right) $ which is called Dirichlet type $%
q$-Euler numbers. Then, by (\ref{Equation 25}), we easily derive the
following%
\begin{equation}
\widetilde{\mathcal{E}}_{n,q}^{\chi }\left( x\mid \alpha :\beta \right)
=\sum_{l=0}^{n}\binom{n}{l}q^{\alpha lx}\widetilde{\mathcal{E}}_{l,q}^{\chi
}\left( \alpha :\beta \right) \left[ x:q^{\alpha }\right] ^{n-l}\text{.}
\label{Equation 26}
\end{equation}

\begin{theorem}
The following equality holds true:%
\begin{eqnarray*}
\widetilde{\mathcal{E}}_{n,q}^{\chi }\left( x\mid \alpha :\beta \right)
&=&\sum_{l=0}^{n}\binom{n}{l}\widetilde{\mathcal{E}}_{l,q}^{\chi }\left(
\alpha :\beta \right) \sum_{j=0}^{l}\binom{l}{j}\left( q^{\alpha }-1\right)
^{j}\left( n-l+j\right) !\left( -1\right) ^{n-l+j} \\
&&\times \sum_{m,n=0}^{\infty }\binom{n-l+j+m-1}{m}\alpha ^{n}q^{\alpha
m}\left( \log q\right) ^{n}\mathcal{S}\left( n,n-l+j\right) \frac{x^{n}}{n!}%
\text{.}
\end{eqnarray*}
\end{theorem}

\begin{proof}
To prove this, by applying (\ref{Equation 26}), we easily discover the
following assertion%
\begin{equation*}
\widetilde{\mathcal{E}}_{n,q}^{\chi }\left( x\mid \alpha :\beta \right)
=\sum_{l=0}^{n}\binom{n}{l}\widetilde{\mathcal{E}}_{l,q}^{\chi }\left(
\alpha :\beta \right) \sum_{j=0}^{l}\binom{l}{j}\left( q^{\alpha }-1\right)
^{j}\left[ x:q^{\alpha }\right] ^{n-l+j}\text{.}
\end{equation*}%
The Second kind Stirling numbers are defined by means of the following
generating function.%
\begin{equation}
\sum_{n=0}^{\infty }\mathcal{S}\left( n,k\right) \frac{t^{n}}{n!}=\frac{%
\left( e^{t}-1\right) ^{k}}{k!}  \label{Equation 27}
\end{equation}%
(for details on this subject, see \cite{Kim 12}). $t$ replace by $\alpha
x\log q$ in (\ref{Equation 27}), then, we easily derive the following%
\begin{equation}
\left[ x:q^{\alpha }\right] ^{k}=k!\left( -1\right) ^{k}\sum_{m,n=0}^{\infty
}\binom{k+m-1}{m}\alpha ^{n}q^{\alpha m}\left( \log q\right) ^{n}\mathcal{S}%
\left( n,k\right) \frac{x^{n}}{n!}\text{,}  \label{Equation 28}
\end{equation}%
where $\sum_{m,n=0}^{\infty }=\sum_{m=0}^{\infty }\sum_{n=0}^{\infty }$.
Thus, by (\ref{Equation 27}) and (\ref{Equation 28}), we get the desired
result and proof is complete.
\end{proof}

\begin{theorem}
The following equality%
\begin{equation*}
\widetilde{\mathcal{E}}_{n,q}^{\chi }\left( x\mid \alpha :\beta \right) =%
\frac{\left[ d:q^{\alpha }\right] ^{n}}{\left[ d:-q^{\beta }\right] }%
\sum_{a=0}^{d-1}\left( -q\right) ^{a}\chi \left( a\right) \widetilde{E}%
_{n,q^{d}}\left( \frac{x+a}{d}\mid \alpha :\beta \right)
\end{equation*}%
holds true.
\end{theorem}

\begin{proof}
By applying the $p$-adic integral representation on the Dirichlet type of $q$%
-Euler polynomials with weight $\alpha$ and $\beta$, we compute as follows:%
\begin{eqnarray*}
\widetilde{\mathcal{E}}_{n,q}^{\chi }\left( x\mid \alpha :\beta \right)
&=&\int_{%
%TCIMACRO{\U{2124} }%
%BeginExpansion
\mathbb{Z}
%EndExpansion
_{p}}\chi \left( \eta \right) \left[ x+\eta :q^{\alpha }\right] ^{n}d\mu
_{-q^{\beta }}\left( \eta \right) \\
&=&\lim_{n\rightarrow \infty }\frac{1}{\left[ dp^{n}:-q^{\beta }\right] }%
\sum_{y=0}^{dp^{n}-1}\left( -q\right) ^{y}\chi \left( y\right) \left[
x+y:q^{\alpha }\right] \\
&=&\frac{1}{\left[ d:-q^{\beta }\right] }\lim_{n\rightarrow \infty }\frac{1}{%
\left[ p^{n}:-q^{d\beta }\right] }\sum_{y=0}^{p^{n}-1}\sum_{a=0}^{d-1}\left(
-q\right) ^{a+dy}\chi \left( a+dy\right) \left[ x+a+dy:q^{\alpha }\right]
^{n} \\
&=&\frac{\left[ d:q^{\alpha }\right] ^{n}}{\left[ d:-q^{\beta }\right] }%
\sum_{a=0}^{d-1}\left( -q\right) ^{a}\chi \left( a\right) \lim_{n\rightarrow
\infty }\frac{1}{\left[ p^{n}:-q^{d}\right] }\sum_{y=0}^{p^{n}-1}\left(
-q^{d\beta }\right) ^{y}\left[ \frac{x+a}{d}+y:q^{d\alpha }\right] ^{n} \\
&=&\frac{\left[ d:q^{\alpha }\right] ^{n}}{\left[ d:-q^{\beta }\right] }%
\sum_{a=0}^{d-1}\left( -q\right) ^{a}\chi \left( a\right) \widetilde{E}%
_{n,q^{d}}\left( \frac{x+a}{d}\mid \alpha :\beta \right) \text{.}
\end{eqnarray*}

Here, $\widetilde{E}_{n,q^{d}}\left( \frac{x+a}{d}\mid \alpha :\beta \right) 
$ is defined by Ryoo in \cite{Ryoo 3}, which is called $q$-Euler polynomials
with weight $\left( \alpha ,\beta \right) $. As a result, we have the proof
of the theorem.
\end{proof}

\section{\textbf{On }$p$\textbf{-adic Dirichlet type of }$q$\textbf{-Euler
measure with weight }$\protect\alpha $\textbf{\ and }$\protect\beta $}

Now, we introduce a map $\mu _{k,q}^{\left( \alpha ,\beta \right) }\left(
a+p^{n}%
%TCIMACRO{\U{2124} }%
%BeginExpansion
\mathbb{Z}
%EndExpansion
_{p}\right) $ on the balls in $%
%TCIMACRO{\U{2124} }%
%BeginExpansion
\mathbb{Z}
%EndExpansion
_{p}$ as follows:%
\begin{equation}
\mathcal{\mu }_{k,q}^{\left( \alpha ,\beta \right) }\left( a+p^{n}%
%TCIMACRO{\U{2124} }%
%BeginExpansion
\mathbb{Z}
%EndExpansion
_{p}\mid \chi \right) =\frac{\left[ p^{n}:q^{\alpha }\right] ^{k}}{\left[
p^{n}:-q^{\beta }\right] }\chi \left( a\right) \left( -1\right)
^{a}q^{a}f_{k,p^{n}}\left( \frac{\left\{ a\right\} _{n}}{p^{n}}\mid \alpha
:\beta \right)  \label{Equation 29}
\end{equation}

where $\left\{ a\right\} _{n}\equiv a\left( \func{mod}p^{n}\right) $.

\begin{theorem}
Let $\alpha $, $k\in 
%TCIMACRO{\U{2115} }%
%BeginExpansion
\mathbb{N}
%EndExpansion
$. Then we specify that $\mathcal{\mu }_{k,q}^{\left( \alpha ,\beta \right)
} $ is $p$-adic measure on $%
%TCIMACRO{\U{2124} }%
%BeginExpansion
\mathbb{Z}
%EndExpansion
_{p}$ if and only if%
\begin{equation*}
f_{k,q^{p^{n}}}\left( \frac{a}{p^{n}}\mid \alpha :\beta \right) =\frac{\left[
p^{n}:q^{p\alpha }\right] ^{k}}{\left[ p^{n}:-q^{p\beta }\right] }%
\sum_{b=0}^{p-1}\left( -1\right) ^{b}q^{bp^{n}}f_{k,\left( q^{p^{n}}\right)
^{p}}\left( \frac{\frac{a}{p^{n}}+b}{p}\mid \alpha :\beta \right) \text{.}
\end{equation*}
\end{theorem}

\begin{proof}
By similar method in \cite{Jolany}, we can state the proof of this theorem.
Therefore, we omit it.
\end{proof}

We now set as follows:%
\begin{equation}
f_{k,q^{p^{n}}}\left( \frac{a}{p^{n}}\mid \alpha :\beta \right) =\widetilde{E%
}_{n,q^{p^{n}}}\left( \frac{a}{p^{n}}\mid \alpha :\beta \right) \text{.}
\label{Equation 30}
\end{equation}

From (\ref{Equation 29}) and (\ref{Equation 30}), we easily see%
\begin{equation}
\mathcal{\mu }_{k,q}^{\left( \alpha ,\beta \right) }\left( a+p^{n}%
%TCIMACRO{\U{2124} }%
%BeginExpansion
\mathbb{Z}
%EndExpansion
_{p}\mid \chi \right) =\frac{\left[ p^{n}:q^{\alpha }\right] ^{k}}{\left[
p^{n}:-q^{\beta }\right] }\chi \left( a\right) \left( -1\right) ^{a}q^{a}%
\widetilde{E}_{k,q^{p^{n}}}\left( \frac{a}{p^{n}}\mid \alpha :\beta \right) 
\text{.}  \label{Equation 31}
\end{equation}

By (\ref{Equation 1}) and (\ref{Equation 31}), then, we have the following
theorem.

\begin{theorem}
For $\alpha $,$k\in 
%TCIMACRO{\U{2115} }%
%BeginExpansion
\mathbb{N}
%EndExpansion
$, we have%
\begin{equation*}
\int_{X}d\mu _{k,q}^{\left( \alpha ,\beta \right) }\left( x\mid \chi \right)
=\widetilde{\mathcal{E}}_{k,q}^{\chi }\left( \alpha :\beta \right) \text{.}
\end{equation*}
\end{theorem}

\begin{proof}
By using combinatorial techniques, we compute as follows:%
\begin{eqnarray*}
\int_{X}d\mu _{k,q}^{\left( \alpha ,\beta \right) }\left( x\mid \chi \right)
&=&\lim_{n\rightarrow \infty }\frac{\left[ dp^{n}:q^{\alpha }\right] ^{k}}{%
\left[ dp^{n}:-q^{\beta }\right] }\sum_{x=0}^{dp^{n}-1}\chi \left( x\right)
\left( -1\right) ^{x}q^{x}\widetilde{E}_{k,q^{dp^{n}}}\left( \frac{x}{dp^{n}}%
\mid \alpha :\beta \right) \\
&=&\frac{\left[ d:q^{\alpha }\right] ^{k}}{\left[ d:-q^{\beta }\right] }%
\sum_{a=0}^{d-1}\left( -1\right) ^{a}q^{a}\chi \left( a\right)
\lim_{n\rightarrow \infty }\frac{\left[ p^{n}:q^{d\alpha }\right] ^{k}}{%
\left[ p^{n}:-q^{d\beta }\right] }\sum_{x=0}^{p^{n}-1}\left( -1\right)
^{x}q^{dx}\widetilde{E}_{k,\left( q^{d}\right) ^{p^{n}}}\left( \frac{\frac{a%
}{d}+x}{p^{n}}\mid \alpha :\beta \right) \\
&=&\frac{\left[ d:q^{\alpha }\right] ^{k}}{\left[ d:-q^{\beta }\right] }%
\sum_{a=0}^{d-1}\left( -1\right) ^{a}q^{a}\chi \left( a\right) \widetilde{E}%
_{n,q^{d}}\left( \frac{a}{d}\mid \alpha :\beta \right)
\end{eqnarray*}

so,we obtain the desired result.
\end{proof}

\begin{theorem}
For any $k\in 
%TCIMACRO{\U{2115} }%
%BeginExpansion
\mathbb{N}
%EndExpansion
$, we get%
\begin{equation*}
\int_{pX}d\mu _{k,q}^{\left( \alpha ,\beta \right) }\left( x\mid \chi
\right) =\chi \left( p\right) \frac{\left[ p:q^{\alpha }\right] }{\left[
p:-q^{\beta }\right] }\widetilde{\mathcal{E}}_{k,q^{p}}^{\chi }\left( \alpha
:\beta \right) \text{.}
\end{equation*}
\end{theorem}

\begin{proof}
From (\ref{Equation 1}) and (\ref{Equation 31}), we derive the followings
assertions%
\begin{eqnarray*}
\int_{pX}d\mu _{k,q}^{\left( \alpha ,\beta \right) }\left( x\mid \chi
\right) &=&\lim_{n\rightarrow \infty }\frac{\left[ dp^{n+1}:q^{\alpha }%
\right] ^{k}}{\left[ dp^{n+1}:-q^{\beta }\right] }\sum_{x=0}^{dp^{n}-1}\chi
\left( px\right) \left( -1\right) ^{px}q^{px}\widetilde{E}%
_{k,q^{dp^{n}}}\left( \frac{px}{dp^{n+1}}\mid \alpha :\beta \right) \\
&=&\chi \left( p\right) \frac{\left[ p:q^{\alpha }\right] }{\left[
p:-q^{\beta }\right] }\frac{\left[ d:q^{p\alpha }\right] ^{k}}{\left[
d:-q^{p\beta }\right] }\sum_{a=0}^{d-1}\left\{ 
\begin{array}{c}
\left( -1\right) ^{a}q^{pa}\chi \left( a\right) \lim_{n\rightarrow \infty }%
\frac{\left[ p^{n}:q^{dp\alpha }\right] ^{k}}{\left[ p^{n}:-q^{pd\beta }%
\right] } \\ 
\times \sum_{x=0}^{p^{n}-1}\left( -1\right) ^{x}q^{pdx}\widetilde{E}%
_{k,\left( q^{d}\right) ^{p^{n}}}\left( \frac{dp\left( \frac{a}{d}+x\right) 
}{pdp^{n}}\mid \alpha :\beta \right)%
\end{array}%
\right\} \\
&=&\chi \left( p\right) \frac{\left[ p:q^{\alpha }\right] }{\left[
p:-q^{\beta }\right] }\frac{\left[ d:q^{p\alpha }\right] ^{k}}{\left[
d:-q^{p\beta }\right] }\sum_{a=0}^{d-1}\left( -1\right) ^{a}q^{pa}\chi
\left( a\right) \widetilde{E}_{k,q^{pd}}\left( \frac{a}{d}\mid \alpha :\beta
\right) \\
&=&\chi \left( p\right) \frac{\left[ p:q^{\alpha }\right] }{\left[
p:-q^{\beta }\right] }\widetilde{\mathcal{E}}_{k,q^{p}}^{\chi }\left( \alpha
:\beta \right) \text{.}
\end{eqnarray*}

Thus, we get the desired result and proof is complete.
\end{proof}

By the same method which we used in above theorem, by a little bit
manipulations we can state the following theorem.

\begin{theorem}
For $c\left( \neq 1\right) \in X^{\ast }$\bigskip , we have%
\begin{equation*}
\int_{pX}d\mu _{k,q^{\frac{1}{c}}}^{\left( \alpha ,\beta \right) }\left(
cx\mid \chi \right) =\chi \left( \frac{p}{c}\right) \frac{\left[ p:q^{\frac{%
\alpha }{c}}\right] }{\left[ p:\left( -q^{\beta }\right) ^{\frac{1}{c}}%
\right] }\widetilde{\mathcal{E}}_{k,q^{\frac{p}{c}}}^{\chi }\left( \alpha
:\beta \right) \text{.}
\end{equation*}
\end{theorem}

\begin{theorem}
For $c\left( \neq 1\right) \in X^{\ast }$\bigskip , we have%
\begin{equation*}
\int_{X}d\mu _{k,q^{\frac{1}{c}}}^{\left( \alpha ,\beta \right) }\left(
cx\mid \chi \right) =\chi \left( \frac{1}{c}\right) \widetilde{\mathcal{E}}%
_{k,q^{\frac{1}{c}}}^{\chi }\left( \alpha :\beta \right) \text{.}
\end{equation*}
\end{theorem}

We can define the following identity:%
\begin{equation*}
\mathcal{\mu }_{k,c,q}^{\left( \alpha ,\beta \right) }\left( U\mid \chi
\right) =\mathcal{\mu }_{k,q}^{\left( \alpha ,\beta \right) }\left( U\mid
\chi \right) -c^{-1}\frac{\left[ c^{-1}:q^{\alpha }\right] ^{k}}{\left[
c^{-1}:-q^{\beta }\right] }\mathcal{\mu }_{k,q^{\frac{1}{c}}}^{\left( \alpha
,\beta \right) }\left( cU\mid \chi \right)
\end{equation*}

here $U$ is any compact open subset of $%
%TCIMACRO{\U{2124} }%
%BeginExpansion
\mathbb{Z}
%EndExpansion
_{p}$, it can be written as a finite disjoint union of sets%
\begin{equation*}
U=\overset{k}{\underset{j=1}{\cup }}\left( a_{j}+p^{n}%
%TCIMACRO{\U{2124} }%
%BeginExpansion
\mathbb{Z}
%EndExpansion
_{p}\right) ,
\end{equation*}

where $n\in 
%TCIMACRO{\U{2115} }%
%BeginExpansion
\mathbb{N}
%EndExpansion
$ and $a_{1},a_{2},...,a_{k}\in 
%TCIMACRO{\U{2124} }%
%BeginExpansion
\mathbb{Z}
%EndExpansion
$ with $0\leq a_{i}<p^{n}$ for $i=1,2,...,k$.

\begin{theorem}
For $c\left( \neq 1\right) \in X^{\ast }$\bigskip , we procure the following%
\begin{equation*}
\int_{X^{\ast }}d\mu _{k,c,q}^{\left( \alpha ,\beta \right) }\left( cx\mid
\chi \right) =(1-\chi ^{p})\left( 1-c^{-1}\chi ^{c^{-1}}\right) \widetilde{%
\mathcal{E}}_{k,q}^{\chi }\left( \alpha :\beta \right)
\end{equation*}%
where the operator $\chi ^{y}:=\chi ^{y,k,\alpha ;q}$ on $f\left( q\right) $
is defined by%
\begin{equation*}
\chi ^{y}f\left( q\right) =\chi ^{y,k,\alpha ;q}f\left( q\right) =\frac{%
\left[ y:q^{\alpha }\right] }{\left[ y:-q^{\beta }\right] }\chi \left(
y\right) f\left( q^{y}\right)
\end{equation*}%
That is, we can write%
\begin{equation*}
\chi ^{x,k,\alpha ;q}\circ \chi ^{y,k,\alpha ;q}f\left( q\right) =\chi
^{xy,k,\alpha ;q}f\left( q\right) =\chi ^{xy}f\left( q\right) \text{.}
\end{equation*}
\end{theorem}

\begin{proof}
$\bigskip $To prove this, we assume that $f\left( q\right) =\widetilde{%
\mathcal{E}}_{k,q}^{\chi }\left( \alpha :\beta \right) $. Then, we get%
\begin{equation*}
\left\{ 
\begin{array}{c}
\widetilde{\mathcal{E}}_{k,q}^{\chi }\left( \alpha :\beta \right) -\chi
\left( p\right) \frac{\left[ p:q^{\alpha }\right] }{\left[ p:-q^{\beta }%
\right] }\widetilde{\mathcal{E}}_{k,q^{p}}^{\chi }\left( \alpha :\beta
\right) -c^{-1}\frac{\left[ c^{-1}:q^{\alpha }\right] ^{k}}{\left[
c^{-1}:-q^{\beta }\right] } \\ 
\times \chi \left( \frac{1}{c}\right) \widetilde{\mathcal{E}}_{k,q^{\frac{1}{%
c}}}^{\chi }\left( \alpha :\beta \right) +\chi \left( \frac{p}{c}\right) 
\frac{\left[ p:q^{\frac{\alpha }{c}}\right] }{\left[ p:\left( -q^{\beta
}\right) ^{\frac{1}{c}}\right] }\widetilde{\mathcal{E}}_{k,q^{\frac{p}{c}%
}}^{\chi }\left( \alpha :\beta \right)%
\end{array}%
\right\}
\end{equation*}

From this, we derive the following%
\begin{equation*}
(1-\chi ^{p})\left( 1-c^{-1}\chi ^{c^{-1}}\right) \widetilde{\mathcal{E}}%
_{k,q}^{\chi }\left( \alpha :\beta \right)
\end{equation*}

Then, we complete the proof of theorem.
\end{proof}

\section{$\protect\bigskip $\textbf{Analytic continuation of }$q$\textbf{%
-Euler Polynomials with weight }$\protect\alpha $}

The concept of analytic continuation just means enlarging the domain without
giving up the property of being differentiable, i.e. holomorphic or
meromorphic. More precisely Let $f_1$ and $f_2$ be analytic functions on
domains $\Omega_1$ and $\Omega_2,$ respectively, and suppose that the
intersection $\Omega_1\cap\Omega_2$ is not empty and that$f_1 =f_2 $ on $%
\Omega_1\cap\Omega_2$. Then $f_2$ is called an analytic continuation of $f_1$
to $\Omega_2$, and vice versa . Moreover, if it exists, the analytic
continuation of $f_1$ to $\Omega_2 $ is unique. By means of analytic
continuation, starting from a representation of a function by any one power
series, any number of other power series can be found which together define
the value of the function at all points of the domain. Furthermore, any
point can be reached from a point without passing through a singularity of
the function, and the aggregate of all the power series thus obtained
constitutes the analytic expression of the function. So we are ready to
state analytic continuation of $q$-Euler polynomials with weight $\alpha$ as
follows.\newline

For coherence with the redefinition of $\widetilde{E}_{n,q}\left( \alpha
\right) =\widetilde{E}_{q}\left( n:\alpha \right) $, we have%
\begin{equation*}
\widetilde{E}_{n,q}\left( x\mid \alpha \right) =q^{-\alpha x}\sum_{k=0}^{n}%
\binom{n}{k}q^{\alpha kx}\widetilde{E}_{k,q}\left( \alpha \right) \left[
x:q^{\alpha }\right] ^{n-k}\text{.}
\end{equation*}

Let $\Gamma \left( s\right) $ be Euler-gamma function. Then the analytic
continuation can be get as 
\begin{eqnarray*}
n &\mapsto &s\in 
%TCIMACRO{\U{211d} }%
%BeginExpansion
\mathbb{R}
%EndExpansion
\text{, }x\mapsto w\in 
%TCIMACRO{\U{2102} }%
%BeginExpansion
\mathbb{C}
%EndExpansion
\text{,} \\
\widetilde{E}_{n,q}\left( \alpha \right) &\mapsto &\widetilde{E}_{q}\left(
k+s-\left[ s\right] :\alpha \right) =\widetilde{\zeta }_{E,q}\left( -\left(
k+s-\left[ s\right] \right) \mid \alpha \right) \text{,} \\
\binom{n}{k} &=&\frac{\Gamma \left( n+1\right) }{\Gamma \left( n-k+1\right)
\Gamma \left( k+1\right) }\mapsto \frac{\Gamma \left( s+1\right) }{\Gamma
\left( 1+k+\left( s-\left[ s\right] \right) \right) \Gamma \left( 1+\left[ s%
\right] -k\right) } \\
\widetilde{E}_{s,q}\left( w\mid \alpha \right) &\mapsto &\widetilde{E}%
_{q}\left( s,w:\alpha \right) =q^{-\alpha w}\sum_{k=-1}^{\left[ s\right] }%
\frac{\Gamma \left( s+1\right) \widetilde{E}_{q}\left( k+\left( s-\left[ s%
\right] \right) :\alpha \right) q^{\alpha w\left( k+\left( s-\left[ s\right]
\right) \right) }}{\Gamma \left( 1+k+\left( s-\left[ s\right] \right)
\right) \Gamma \left( 1+\left[ s\right] -k\right) }\left[ w:q^{\alpha }%
\right] ^{\left[ s\right] -k} \\
&=&q^{-\alpha w}\sum_{k=0}^{\left[ s\right] +1}\frac{\Gamma \left(
s+1\right) \widetilde{E}_{q}\left( -1+k+\left( s-\left[ s\right] \right)
:\alpha \right) q^{\alpha w\left( k-1+\left( s-\left[ s\right] \right)
\right) }}{\Gamma \left( k+\left( s-\left[ s\right] \right) \right) \Gamma
\left( 2+\left[ s\right] -k\right) }\left[ w:q^{\alpha }\right] ^{\left[ s%
\right] +1-k}\text{.}
\end{eqnarray*}

Here $\left[ s\right] $ gives the integer part of s, and so $s-\left[ s%
\right] $ gives the fractional part.

Deformation of the curve $\widetilde{E}_{q}\left( 1,w:\alpha \right) $ into
the curve of $\widetilde{E}_{q}\left( 2,w:\alpha \right) $ is by means of
the real analytic cotinuation $\widetilde{E}_{q}\left( s,w:\alpha \right) $, 
$1\leq s\leq 2$, $-0.5\leq w\leq 0.5$.\newline

\begin{acknowledgement}
The third author would like to thank the Association SARA-GHU \`{a}
Marseille for their hospitality during his stay there, when the work for
this paper was done and dedicated this paper to Neda Agha-Soltan.\newline
\end{acknowledgement}

\end{document}